\documentclass[a4paper,reqno,10pt]{amsart}
\usepackage{amsfonts}
\usepackage{amsmath}
\usepackage{amssymb}
\usepackage{tikz-cd}
\usepackage{eurosym}
\usepackage{eurosym}
\usepackage{amsfonts}
\usepackage{amsmath}
\usepackage{amssymb}
\usepackage{bbm}
\usepackage{amscd}
\usepackage{cite}
\usepackage{bbm}
\usepackage{amscd}
\usepackage[left=2cm,right=2cm,bottom=3cm,top=3cm]{geometry}
\usepackage{etoolbox}
\usepackage{filecontents}
\usepackage[hidelinks]{hyperref}
\usepackage{enumerate}

\AtBeginDocument{   \def\MR#1{}
}
\newtheorem{theorem}{Theorem}[section]
\newtheorem{proposition}[theorem]{Proposition}
\theoremstyle{definition}
\newtheorem{lemma}[theorem]{Lemma}
\newtheorem{definition}[theorem]{Definition}

\newtheorem{problem}[theorem]{Problem}
\newtheorem{remark}[theorem]{Remark}

\begin{document}
\title[Games orbits play and obstructions to Borel reducibility]{Games
orbits play and obstructions to Borel reducibility}
\author{Martino Lupini}
\address{Mathematics Department, Caltech, 1200 E. California Blvd, MC 253-37
Pasadena, CA 91125, USA}
\email{lupini@caltech.edu}
\urladdr{http://www.lupini.org/}
\author{Aristotelis Panagiotopoulos}
\address{Department of Mathematics, 1409 W. Green St., University of
Illinois, Urbana, IL 61801}
\curraddr{Mathematics Department, Caltech, 1200 E. California Blvd, MC 253-37
Pasadena, CA 91125, USA}
\email{panagio@caltech.edu}
\urladdr{http://www.its.caltech.edu/~panagio/}
\dedicatory{}
\subjclass[2000]{Primary 03E15, 54H20; Secondary 20L05, 54H05}
\thanks{This work was initiated during a visit of A.P.\ at the California
Institute of Technology in the Spring 2016. The authors gratefully
acknowledge the hospitality and the financial support of the Institute.
M.L.\ was partially supported by the NSF Grant DMS-1600186.}
\keywords{Polish group, Polish groupoid, turbulence, orbit equivalence
relation, Borel game, CLI group, non-Archimedean group}

\begin{abstract}
We introduce a new, game-theoretic approach to anti-classification results
for orbit equivalence relations. Within this framework, we give a short
conceptual proof of Hjorth's turbulence theorem. We also introduce a new
dynamical criterion providing an obstruction to classification by orbits of
CLI groups. We apply this criterion to the relation of equality of countable
sets of reals, and the relations of unitary conjugacy of unitary and
selfadjoint operators on the separable infinite-dimensional Hilbert space.
\end{abstract}

\maketitle

\section{Introduction}

Borel complexity theory provides a framework for studying the relative
complexity of classification problems in mathematics. In this context, a
classification problem is regarded as an equivalence relation on a standard
Borel space and \emph{\ Borel reductions} provide the main notion of comparison between equivalence relations. Two distinguished classes of equivalence relations that are often used as
benchmarks to measure the complexity of other equivalence relations are
equality on a Polish space and isomorphism of countable structures. An
equivalence relation $E$ is called \emph{concretely classifiable} if $E$ is
Borel reducible to equality on some Polish space $Y$. More generously, an
equivalence relation is called \emph{classifiable by countable structures }%
if it is Borel reducible to the relation $\cong _{\mathcal{L}}$ of
isomorphism within the class of countably infinite structures in some first
order language $\mathcal{L}$.

Many naturally arising equivalence relations are given as orbit equivalence
relations $E_{G}^{X}$ associated with a continuous action of a Polish group $%
G $ on a Polish space $X$. For instance, the relation $\cong _{\mathcal{L}}$
is induced by the canonical logic action of the group $S_{\infty }$ of
permutations of $\omega $ on the space $\mathrm{Mod}(\mathcal{L})$ of $%
\mathcal{L}$-structures with universe $\omega $. Similarly, the equality
relation on a Polish space is induced by the action of the trivial group.

In the process of determining the exact complexity of an equivalence relation  one often has to establish both positive and negative results. For the later, it is important to isolate criteria that imply nonclassifiability by certain invariants. The well-known
criterion of \emph{generic ergodicity }(having meager orbits and a dense
orbit) provides a dynamical condition on a continuous action of a Polish
group on a Polish space ensuring that the corresponding orbit equivalence
relation is not concretely classifiable. Hjorth's notion of \emph{turbulent
action} is a strengthening of generic ergodicity, ensuring that the
associated orbit equivalence relation is not classifiable by countable
structures. Both these results can actually be seen as addressing the following
general problem.

\begin{problem}
\label{problem} Given a class of Polish groups $\mathcal{C}$, which
dynamical conditions on a Polish $G$-space $X$ ensure that the corresponding
orbit equivalence relation is not Borel reducible to $E_{H}^{Y}$ for some
Borel action of a Polish group $H$ in $\mathcal{C}$ on a Polish space $Y$?
\end{problem}

Indeed, generic ergodicity provides such a criterion for the class $\mathcal{%
C}$ of \emph{compact Polish groups }\cite[Proposition 6.1.10]%
{gao_invariant_2009}. Hjorth's turbulence theorem \cite%
{hjorth_classification_2000} addresses this problem in the case when $%
\mathcal{C}$ is the class of \emph{non-Archimedean Polish groups}.
Turbulence has played a key role in Borel complexity theory in the last two
decades and it is to this day essentially the only known method to prove
unclassifiability by countable structures; see \cite%
{hjorth_non-smooth_1997,farah_turbulence_2014,foreman_anti-classification_2004,lupini_unitary_2014,kerr_borel_2015,kerr_turbulence_2010,ando_borel_2015, ando_weylvon_2015,sasyk_borel_2009,sasyk_classification_2009,sasyk_turbulence_2010,hartz_classification_2015,tornquist_set_2012,ioana_subequivalence_2009,kechris_complexity_2012,kechris_global_2010,gao_invariant_2009,hjorth_classification_2000,kechris_actions_2002,farah_dichotomy_2012}%
. There has been so far little progress into obtaining similar criteria for
other interesting classes of Polish groups.

The purpose of this paper is two-fold. Our first goal is to introduce a
game-theoretic approach to Problem \ref{problem}. This approach consists in
endowing the space $X/G$ of orbits of a Polish $G$-space $X$ with different
graph structures, and then showing that a Baire measurable $%
(E_{G}^{X},E_{H}^{Y})$-homomorphism $f\colon X\rightarrow Y$ induces a graph
homomorphism $X/G\rightarrow Y/G$ after restricting to an invariant dense $%
G_{\delta }$ set. This perspective allows us to give a short conceptual
proof of Hjorth's turbulence theorem, avoiding the substantial amount of
bookkeeping of Hjorth's original argument \cite{hjorth_classification_2000};
see also \cite[Chapter 10]{gao_invariant_2009}.

The second goal of this paper is to use the above-mentioned game-theoretic
approach to address Problem \ref{problem} for the class of \emph{CLI groups}%
. Recall that a CLI group is a Polish group that admits a \emph{compatible
complete left-invariant metric}. Every locally compact group, as well as
every solvable Polish group---in particular, every abelian Polish group---is
CLI \cite[Corollary 3.7]{hjorth_vaughts_1999}. This class of groups has been
considered in several papers so far. For instance, \cite[Corollary 5.C.6]%
{becker_polish_1998} settled the topological Vaught conjecture for CLI
groups. It is also proved in \cite[Theorem 5.B.2]{becker_polish_1998} that
CLI groups satisfy an analog of the Glimm-Effros dichotomy. In \cite[Theorem
1.1]{gao_automorphism_1998} it is shown that the non-Archimedean CLI groups
are precisely the automorphism groups of countable structures whose Scott
sentence does not have an uncountable model.\ The class of CLI groups has
been further studied in \cite{malicki_polish_2011}, where it is shown that
it forms a coanalytic non-Borel subset of the class of Polish groups.

A fundamental tool in the study of dichotomies for orbit equivalence
relations from \cite{becker_polish_1998} is the notion of (left) $\iota $%
-embeddability for points in a Polish $G$-space. We work here with the right
variant of $\iota $-embeddability which we call \emph{Becker embeddability}.
We prove that a Baire-measurable homomorphism between orbit equivalence
relations necessarily preserves Becker embeddability on an invariant dense $%
G_{\delta }$ set. From this we extract in Theorem \ref{Theorem:criterion} a
dynamical condition which answers Problem \ref{problem} for the class of CLI
groups. We then apply it to show that the Friedman-Stanley jump of equality $%
=^{+}$ is not Borel reducible to the orbit equivalence relation induced by a
Borel action of a CLI group. The only proof of this fact that we are aware
of relies on meta-mathematical reasoning and involves the theory of pinned
equivalence relations; see \cite{kanovei_borel_2008}. A natural reduction
from this relation to the relations of unitary equivalence of bounded
unitary or selfajdoint operators on an infinite-dimensional Hilbert space,
shows that the latter relations are also not classifiable by the orbits of a
CLI group actions. We note that it is still an open question if an action of
the unitary group can induce an orbit equivalence relation which is
universal for orbit equivalence relations induced by Polish group actions.
Our results show that the complexity of possible orbit equivalence relations
of $\mathcal{U}(\mathcal{H})$-actions is not bounded from above by the
complexity of orbit equivalence relations induced by continuous CLI group
actions.

We conclude by discussing how all the results of the present paper admit
natural generalizations from Polish group actions to Polish groupoids.
Turbulence theory for Polish groupoids has been developed in \cite%
{hartz_classification_2015}. Applications of this more general framework to
classification problems in operator algebras have also been presented in 
\cite{hartz_classification_2015}.

Besides this introduction, the present paper is divided into three sections.
In Section \ref{Section:CLI} we present the results about
Becker-embeddability and CLI groups. In Section \ref{Section:turbulence} we
present the short and conceptual proof of Hjorth's turbulence theorem
mentioned above. Finally in Section \ref{Section:groupoids} we recall the
fundamental notions about Polish groupoids, and explain how the main results
of this paper can be adapted to this more general setting. In the rest of
the paper, we will use the category quantifier $\forall ^{\ast }x\in U$ for
the statement \textquotedblleft for a comeager set of $x\in U$%
\textquotedblright ; see \cite[Section 3.2]{gao_invariant_2009}.

\subsubsection*{Acknowledgments}

We are grateful to  Samuel Coskey, Alex Kruckman, Alexander Kechris, and S\l awomir Solecki
for many helpful suggestions and remarks. Furthermore, we thank the
anonymous referee for their careful reading of the manuscript and for many
useful comments.

\section{Nonreducibility to CLI group actions\label{Section:CLI}}

\subsection{The Becker-embedding game\label{Sbs:embeddings}}

Recall that a CLI\emph{\ }group is a Polish group that admits a compatible 
\emph{complete left-invariant metric}. It is easy to see that a Polish group
is CLI if and only if it admits a compatible complete right-invariant
metric; see \cite[3.A.2. Proposition]{becker_polish_1998}. Throughout this
section, we will let $G$ be a Polish group, and $X$ be a Polish $G$-space.
The main goal of this section is to provide a dynamical criterion for a
Polish $G$-space ensuring that the corresponding orbit equivalence relation
is not Borel reducible to the orbit equivalence relation induced by a Borel
action of a CLI group. A characterization of CLI groups in terms of tameness
of the corresponding orbit equivalence relations has been obtained in \cite%
{thompson_metamathematical_2006}.

The notion of $\iota $-embeddability for points of $X$ has been introduced
in \cite[Definition 3.D.1]{becker_polish_1998}. Here we will refer to it as 
\emph{left }$\iota $-embeddability, to distinguish it from its natural right
analogue. Recall that a sequence $\left( g_{n}\right) $ in $G$ is Cauchy
with respect to some left-invariant metric on $G$ if and only if it is
Cauchy with respect to every left-invariant metric on $G$ \cite[Proposition
3.B.1]{becker_polish_1998}. If this holds, we say that $\left( g_{n}\right) $
is \emph{left Cauchy}. We define when $\left( g_{n}\right) $ is right Cauchy
in a similar way, by replacing left-invariant metrics with right-invariant
metrics.

\begin{definition}
\label{Definition:iota}Fix $x,y\in X$. We say that $x$ is \emph{left} $\iota 
$-\emph{embeddable} into $y$ if there exists a left Cauchy sequence $\left(
h_{n}\right) _{n\in \omega }$ in $G$ such that $h_{n}x\rightarrow y$. We say that $x$ is 
\emph{right }$\iota $-\emph{embeddable }into $y$ if there exists a right
Cauchy sequence $\left( h_{n}\right) _{n\in \omega }$ in $G$ such that $%
h_{n}y\rightarrow x$.
\end{definition}

In the rest of the paper, we will focus on the notion of right $\iota $%
-embeddability. Similar results can be proved for left $\iota $%
-embeddability. It is easy to see as in the proof of \cite[Proposition 3.D.4]%
{becker_polish_1998} that the relation of right $\iota $-embeddability is a
preorder. Furthermore, if $x$ is right $\iota $-embeddable into $y$, $%
x^{\prime }$ belongs to the $G$-orbit of $x$, and $y^{\prime }$ belongs to
the $G$-orbit of $y$, then $x^{\prime }$ is right $\iota $-embeddable into $%
y^{\prime }$.

We now consider a natural game between two players, and show that it
captures the notion of right $\iota $-embeddability from Definition \ref%
{Definition:iota}. A natural variation of the same game captures the notion
of left $\iota $-embeddability.

\begin{definition}
\label{Becker-game}Suppose that $X$ is a Polish $G$-space, and $x,y\in X$.
We consider the \emph{Becker-embedding game} $\mathrm{Emb}(x,y)$ played
between two players as follows. Set $U_{0}=X$ and $V_{0}=G$.

\begin{enumerate}
\item In the first turn, Player I plays an open neighborhood $U_{1}$ of $x$,
and an open neighborhood $V_{1}$ of the identity of $G$. Player II replies
with an element $g_{0}$ in $V_0$.

\item In the second turn, Player I then plays an open neighborhood $U_{2}$
of $x$, and an open neighborhood $V_{2}$ of the identity of $G$, and Player
II replies with an element $g_{1}$ in $V_{1}$.

\item[(n)] At the $n$-th turn, Player I plays an open neighborhood $U_{n}$
of $x$, and an open neighborhood $V_{n}$ of the identity of $G$, and Player
II responds with an element $g_{n-1}$ in $V_{n-1}$.
\end{enumerate}

The game proceed in this way, producing a sequence $\left( g_{n}\right) $ of
elements of $G$, a sequence $\left( U_{n}\right) $ of open neighborhoods of $%
x$ in $X$, and a sequence $\left( V_{n}\right) $ of open neighborhoods of
the identity in $G$. Player II wins the game if for every $n>0$, $%
g_{n-1}\cdots g_{0}y\in U_{n}$. We say that $x$ is \emph{Becker embeddable}
into $y$---and write $x\preccurlyeq _{\mathcal{B}}y$---if Player II has a
winning strategy for the game $\mathrm{Emb}(x,y)$.
\end{definition}

\begin{lemma}
\label{Lemma:comeager-Becker}If Player II has a winning strategy for the
Becker-embedding game as described in Definition \ref{Becker-game}, then it
also has a winning strategy for the same game with the additional winning
conditions that $g_{n}$ belongs to some given comeager subset $V_{n}^{\ast }$
of $V_{n}$, and $g_{n-1}\cdots g_{0}y$ belongs to some given comeager subset 
$X_{0}$ of $X$, provided that the set of $g\in G$ such that $gy\in X_{0}$ is
comeager.
\end{lemma}

\begin{proof}
Suppose that Player II has a winning strategy for the Becker-embedding
game $\mathrm{Emb}\left( x,y\right) $. The strategy consists of one function $%
g_{n}\left( U_{1},\ldots ,U_{n+1},V_{1},\ldots ,V_{n+1}\right) $ for every $n\geq
0 $, where $U_{0},\ldots ,U_{n+1}$ are open neighborhoods of $x$ and $%
V_{1},\ldots ,V_{n}$ are open neighborhoods of the identity in $G$. Since the strategy is winning, one has that $g_{n-1}\cdots g_{0}y\in U_{n}$ and $g_{n}\in V_{n}$
for every $n\geq 0$.

Let us consider now a run of the game $\mathrm{Emb}\left( x,y\right) $.
Suppose that Player I at turn 1 plays open sets $U_{1},V_{1}$. Consider an
open neighborhood $\widetilde{V}_{1}$ of $1$ such that $\widetilde{V}_{1}%
\widetilde{V}_{1}\subseteq V_{1}$. Let $g_{0}^{\prime
}:=g_{0}(U_{1},\widetilde{V}_{1})$ obtained by applying the given winning
strategy of Player II. Using the assumption in the statement of the lemma we can find $%
\varepsilon _{0}\in G$ such that the following conditions are satisfied:

\begin{itemize}
\item $\varepsilon _{0}^{-1}\in \widetilde{V}_{1}$;

\item $\varepsilon _{0}g_{0}^{\prime }\in V_{0}^{\ast }$;

\item $\varepsilon _{0}g_{0}^{\prime }y\in X_{0}\cap U_{1}$.
\end{itemize}

We can then define $\widetilde{g}_{0}:=\varepsilon _{0}g_{0}^{\prime }$.
Observe that $\widetilde{g}_{0}\in V_{0}^{\ast }$ and $\widetilde{g}_{0}y\in
X_{0}\cap U_{1}$.

Suppose inductively that in the first $n$ turns of the game Player I has played open sets $U_{1},V_{1},\ldots
,U_{n},V_{n}$. Consider the elements $g_{i}^{\prime }:=g_{i}(U_{1},\ldots
,U_{i+1},\widetilde{V}_{1},\ldots ,\widetilde{V}_{i+1})$ of $G$ for $i\in
\left\{ 0,1,\ldots ,n-1\right\} $ produced by applying the original strategy
for Player II, where $\widetilde{V}_{i}$ is an open neighborhood of the
identity in $G$ such that $\widetilde{V}_{i}\widetilde{V}_{i}\subseteq V_{i}$
for $i\in \left\{ 1,2,\ldots ,n\right\} $. We suppose furthermore that we
have defined $\varepsilon _{0},\ldots ,\varepsilon _{n-1},\widetilde{g}%
_{0},\ldots ,\widetilde{g}_{n-1}\in G$ are such that the following
conditions are satisfied for every $i\in \left\{ 0,1,\ldots ,n-1\right\} $:

\begin{itemize}
\item $\varepsilon _{i}^{-1}\in \widetilde{V}_{i+1}$;

\item $\varepsilon _{i}g_{i}^{\prime }\in V_{i}^{\ast }$;

\item $\varepsilon _{i}g_{i}^{\prime }g_{i-1}^{\prime }\cdots g_{0}^{\prime
}y\in X_{0}\cap U_{i+1}$;

\item $\widetilde{g}_{i}=\varepsilon _{i}g_{i}\varepsilon _{i-1}^{-1}$ where 
$\varepsilon _{-1}=1$.
\end{itemize}

Suppose that Player I plays open sets $U_{n+1},V_{n+1}$ at turn $n$.
Consider then $g_{n}^{\prime }:=g_{n}(U_{1},\ldots ,U_{n+1},\widetilde{V}%
_{1},\ldots ,\widetilde{V}_{n+1})$ obtained by applying the original winning
strategy of Player II. Thus we have that $g_{n}^{\prime }\cdots
g_{0}^{\prime }y\in X_{0}\cap U_{n+1}$ and that $g_{n}^{\prime }\varepsilon
_{n-1}^{-1}\in \widetilde{V}_{n}\widetilde{V}_{n}\subseteq V_{n+1}$. By
applying the assumption in the statement of the lemma we can find $\varepsilon _{n}\in G$ such that the
following conditions are satisfied:

\begin{itemize}
\item $\varepsilon _{n}^{-1}\in \widetilde{V}_{n+1}$;

\item $\varepsilon _{n}g_{n}^{\prime }\varepsilon _{n-1}^{-1}\in V_{n}^{\ast
}$;

\item $\varepsilon _{n}g_{n}^{\prime }\cdots g_{0}^{\prime }y\in X_{0}\cap
U_{n+1}$.
\end{itemize}

We then set $\widetilde{g}_{n}:=\varepsilon _{n}g_{n}^{\prime }$. Observe
that $\widetilde{g}_{n}\in V_{n}^{\ast }$ and $\widetilde{g}_{n}\widetilde{g}%
_{n-1}\cdots \widetilde{g}_{0}y=\varepsilon _{n}g_{n}\cdots g_{0}y\in
X_{0}\cap U_{n+1}$.

It is clear from the construction that setting $\widetilde{g}%
_{n}:=g_{n}\left( U_{1},\ldots ,U_{n},V_{1},\ldots ,V_{n}\right) $ gives a
new winning strategy for Player II which satisfies the required additional
conditions.
\end{proof}

We now show that the notion of Becker-embeddability from Definition \ref%
{Becker-game} is actually equivalent to the notion of right $\iota $%
-embeddability from Definition \ref{Definition:iota}.

\begin{lemma}
\label{Lemma Becker embed= winning strat for II} Let $X$ be a Polish $G$%
-space. If $x,y$ are points of $X$, then the following statements are
equivalent:

\begin{enumerate}
\item $x\preccurlyeq _{\mathcal{B}}y$;

\item $x$ is right $\iota $-embeddable in $y$.
\end{enumerate}
\end{lemma}

\begin{proof}
We fix a right-invariant metric $d$ on $G$. For a subset $A$ of $G$ we let $%
\mathrm{diam}\left( A\right) $ be the diameter of $A$ with respect to $d$.

(1)$\Rightarrow $(2) Suppose that Player II has a winning strategy for the
Becker-embedding game $\mathrm{Emb}\left( x,y\right) $. Let Player I play a
sequence $\left( U_{n}\right) $ which forms a basis of open neighborhoods of 
$x$ and a sequence $\left( V_{n}\right) $ which forms a basis of symmetric
open neighborhoods of the identity of $G$ with $\mathrm{diam}\left(
V_{n}\right) <2^{-n}$. Let $\left( g_{n}\right) _{n\in \omega }$ be the
sequence of elements of $G$ given by a winning strategy for Player II. Then
the sequence $\left( h_{n}\right) _{n\in \omega }$ obtained by setting $%
h_{n}:=g_{n}\cdots g_{0}$ is $d$-Cauchy, and $h_{n}y\rightarrow x$.

(2)$\Rightarrow $(1) Suppose that there exists a $d$-Cauchy sequence $\left(
h_{n}\right) _{n\in \omega }$ in $G$ such that $h_{n}y\rightarrow x$. We
describe a winning strategy for Player II. Set $h_{-1}:=1$. Suppose that in
the first turn Player I plays an open neighborhood $U_{1}$ of $x$ and an
open neighborhood $V_{1}$ of the identity of $G$. Player II replies with $%
g_{0}:=h_{k_{0}}$, where $k_{0}\in \omega $ is such so that:

\begin{enumerate}
\item $h_{k}h_{k_{0}}^{-1}\in V_{1}$ for all $k\geq k_{0}$, and

\item $h_{k_{0}}y\in U_{1}$.
\end{enumerate}

The first condition is satisfied by a large enough $k_{0}\in \omega $
because $\left( h_{n}\right) _{n\in \omega }$ is $d$-Cauchy. The second
condition is satisfied by a large enough $k_{0}\in \omega $ because $%
h_{n}x\rightarrow y$. Suppose that in the $n$-turn Player I plays an open
neighborhood $U_{n}$ of $x$ and an open neighborhood $V_{n}$ of the identity
in $G$. Inductively, assume also that $g_{n-2}$ is of the form $%
h_{k_{n-2}}h_{k_{n-3}}^{-1}$ for some $k_{n-2}\in \omega $ such that $%
h_{k}h_{k_{n-2}}^{-1}\in V_{n-1}$ for all $k\geq k_{n-2}$. Player II replies
with $g_{n-1}:=h_{k_{n-1}}h_{k_{n-2}}^{-1}$, where $k_{n-1}\in \omega $ is
such that:

\begin{enumerate}
\item $h_{k}h_{k_{n-1}}^{-1}\in V_{n}$ for all $k\geq k_{n-1}$, and

\item $h_{k_{n-1}}y\in U_{n}$.
\end{enumerate}

Again, our assumptions on the sequence $\left( h_{n}\right) _{n\in \omega }$
guarantee that a large enough $k_{n-1}\in \omega $ satisfies both these
conditions. Then we have that $g_{n-1}\in V_{n-1}$ by the inductive
assumption on $k_{n-2}$. Therefore this procedure describes a winning
strategy for Player II in the Becker-embedding game $\mathrm{Emb}\left(
x,y\right) $.
\end{proof}

We let $X/G$ be the space of $G$-orbits of points of $X$. The
Becker-embeddability preorder defines a directed graph structure on $X/G$
obtained by declaring that there is an arrow from the orbit $\left[ x\right] 
$ of $x$ to the orbit $\left[ y\right] $ of $y$ if and only if $%
x\preccurlyeq _{\mathcal{B}}y$. We will call this the \emph{Becker digraph }$%
\mathcal{B}\left( X/G\right) $ of the Polish $G$-space $X$. Similarly, for a 
$G$-invariant subset $X_{0}$ of $X$ we let $\mathcal{B}\left( X_{0}/G\right) 
$ the induced subgraph of $\mathcal{B}\left( X/G\right) $ only containing
vertices corresponding to orbits from $X_{0}$. Suppose that $G,H$ are Polish
groups, $X$ is a Polish $G$-space, and $Y$ is a Polish $H$-space. Any $%
\left( E_{G}^{X},E_{H}^{Y}\right) $-homomorphism $f:X\rightarrow Y$ induces
a function $\left[ f\right] :X/G\rightarrow Y/H$, $\left[ x\right] \mapsto %
\left[ f(x)\right] $. We will show below that, when $f$ is Baire-measurable,
such a function is \emph{generically }a digraph homomorphism with respect to
the Becker digraph structures on $X/G$ and $Y/H$.

We now describe the notion of Becker-embedding in case of Polish $G$-spaces
arising from classes of countable models. Suppose that $\mathcal{L}=\left(
R_{i}\right) _{i\in I}$ is a countable first order relational language,
where $R_{i}$ is a relation symbol with arity $n_{i}$. Let $\mathrm{Mod}%
\left( \mathcal{L}\right) $ be the space of countable $\mathcal{L}$%
-structures having $\mathbb{N}$ as support, $F$ be a countable fragment of $%
\mathcal{L}_{\omega _{1},\omega }$, and $S_{\infty }$ be the group of
permutations of $\mathbb{N}$. As usual, one can regard $\mathrm{Mod}\left( 
\mathcal{L}\right) $ as the product $\prod\nolimits_{i\in I}2^{(\mathbb{N}%
^{n_{i}})}$. Any $\mathcal{L}_{\omega _{1},\omega }$ formula $\varphi \left(
x_{1},\ldots ,x_{n}\right) $ defines a function $\left[ \varphi \right] :%
\mathrm{Mod}\left( \mathcal{L}\right) \rightarrow \left\{ 0,1\right\} ^{%
\mathbb{N}^{n}}$ given by its interpretation. A set $F$ of $\mathcal{L}%
_{\omega _{1},\omega }$ formulas defines a topology $t_{F}$ on\ $\mathrm{Mod}%
\left( \mathcal{L}\right) $, which is the weakest topology that makes the
functions $\left[ \varphi \right] $ for $\varphi \in F$ continuous. The
canonical action $S_{\infty }\curvearrowright \mathrm{Mod}\left( \mathcal{L}%
\right) $ turns $\left( \mathrm{Mod}\left( \mathcal{L}\right) ,t_{F}\right) $
into a Polish $G$-space. If $x,y\in \mathrm{Mod}\left( \mathcal{L}\right) $,
then we have that $x\preccurlyeq _{\mathcal{B}}y$ if and only if there
exists an injective function $f:\mathbb{N}\rightarrow \mathbb{N}$ that
represents an $F$-embedding from $x$ to $y$, in the sense that\ $f$
preserves the value of formulas $\varphi $ in $F$ with parameters. In the
particular case when $F$ is the collection of atomic first-order formulas,
the topology $t_{F}$ coincides with the product topology, and an $F$%
-embedding is the same as an embedding as $\mathcal{L}$-structure. When $F$
is the collection of all first-order formulas, an $F$-embedding is an
elementary embedding. It is shown in \cite[Proposition 2.D.2]%
{becker_polish_1998}, in the case when $F$ is a \emph{fragment }in the sense
define therein, that the same conclusions holds for \emph{left }$\iota $%
-embeddability.

\subsection{The orbit continuity lemma\label{Section:OrbitContinuity}}

Recall that if $E,F$ are equivalence relations on Polish spaces $X,Y$
respectively, then a $\left( E,F\right) $-homomorphism is a function $%
f:X\rightarrow Y$ mapping $E$-classes to $F$-classes. In this subsection we
isolate a lemma to be used in the rest of the paper. It states that a
Baire-measurable homomorphism between orbit equivalence relations admits a
restriction to a dense $G_{\delta }$ set which is continuous at the level of
orbits, in a suitable sense. Variations of such a lemma are well known. The
starting point is essentially \cite[Lemma 3.17]{hjorth_classification_2000}
modified as in the beginning of the proof of \cite[Theorem 3.18]%
{hjorth_classification_2000}; see also \cite[Lemma 10.1.4 and Theorem 10.4.2]%
{gao_invariant_2009}.

\begin{lemma}
\label{Lemma:orbit-continuity}Suppose that $G,H$ are Polish groups, $X$ is a
Polish $G$-space, and $Y$ is a Polish $H$-space. Let $f:X\rightarrow Y$ be a
Baire-measurable $\left( E_{G}^{X},E_{H}^{Y}\right) $-homomorphism. Then
there exists a dense $G_{\delta }$ subset $C$ of $X$ such that

\begin{itemize}
\item the restriction of $f$ to $C$ is continuous;

\item for any $x\in C$, $\left\{ g\in G:gx\in C\right\} $ is a comeager
subset of $G$;

\item for any $x_{0}\in C$ and for any open neighborhood $W$ of the identity
in $H$ there exists an open neighborhood $U$ of $x_{0}$ and an open
neighborhood $V$ of the identity of $G$ such that for any $x\in U\cap C$ and
for a comeager set of $g\in V$, one has that $f(gx)\in Wf(x)$ and $gx\in C$.
\end{itemize}
\end{lemma}

\begin{proof}
Fix a neighborhood $W_{0}$ of the identity in $H$. We first prove the
following claim: $\forall x_{0}\in X$ $\forall ^{\ast }g_{0}\in G$, there is
an open neighborhood $V$ of the identity in $G$ such that $\forall ^{\ast
}g_{1}\in V$, $f(g_{1}g_{0}x_{0})\in W_{0}f(g_{0}x_{0})$.

Fix a neighborhood $W$ of the identity of $H$ such that $WW^{-1}\subset
W_{0} $. Let $\left( h_{n}\right) $ be a sequence in $H$ such that $\left\{
Wh_{n}:n\in \mathbb{N}\right\} $ is a cover of $H$. Since $Wh_{n}f(x_{0})$
is analytic, the set of elements $x$ of the orbit of $x_{0}$ such that $%
f(x)\in Wh_{n}f(x_{0})$ has the Baire property. Therefore we can find a
sequence $\left( O_{n}\right) $ of open subsets of $G$ with dense union $O$
and a comeager subset $D$ of $O$ such that $\forall g\in D\cap O_{n}$, $%
f(gx_{0})\in Wh_{n}f(x_{0})$. Suppose now that $g_{0}\in D$. Let $n\in 
\mathbb{N}$ be such that $g_{0}\in O_{n}$. Then there exists a neighborhood $%
V$ of the identity of $G$ such that $Vg_{0}\subset O_{n}$. Observe that $%
\left( D\cap O_{n}\right) g_{0}^{-1}\cap V$ is a comeager subset of $V$. If $%
g_{1}\in \left( D\cap O_{n}\right) g_{0}^{-1}\cap V$, then we have%
\begin{equation*}
f(g_{1}g_{0}x_{0})\in Wh_{n}f(x_{0})\text{\quad and\quad }f(g_{0}x_{0})\in
Wh_{n}f(x_{0})\text{.}
\end{equation*}%
Therefore%
\begin{equation*}
f(g_{1}g_{0}x_{0})\in WW^{-1}f(g_{0}x_{0})\subset W_{0}f(g_{0}x_{0})\text{.}
\end{equation*}%
This concludes the proof of the claim.

From the claim and the Kuratowski-Ulam theorem, one deduces that there
exists a dense $G_{\delta }$ subset $C_{0}$ of $X$ such that for every $x\in
C_{0}$ there exists an open neighborhood $V$ of the identity of $G$ such
that $\forall ^{\ast }g\in V$, $f(gx)\in Wf(x)$. Since $f$ is
Baire-measurable, we can furthermore assume that the restriction of $f$ to $%
C_{0}$ is continuous.

Fix now a countable basis $\left( W_{k}\right) $ of open neighborhoods of
the identity of $H$ and a countable basis $\left( V_{n}\right) $ of open
neighborhoods of the identity in $G$. Let $N:X\times \mathbb{N}\rightarrow 
\mathbb{N}\cup \left\{ \infty \right\} $ be the function that assigns to $%
\left( x,k\right) $ the least $n\in \mathbb{N}$ such that $\forall ^{\ast
}g\in V_{n}$, $f(gx)\in W_{k}f(x)$ if such an $n$ exists and $x\in C_{0}$,
and $\infty $ otherwise. Then $N$ is an analytic function, and hence one can
find a dense $G_{\delta }$ subset $C_{1}$ of $X$ contained in $C_{0}$ such
that $N|_{C_{1}\times \mathbb{N}}$ is continuous. By \cite[Proposition 3.2.5
and Theorem 3.2.7]{gao_invariant_2009} the set $C:=\left\{ x\in
C_{1}:\forall ^{\ast }g\in G,gx\in C_{1}\right\} $ is a dense $G_{\delta }$
subset of $X$ such that $\forall x\in C$, $\forall ^{\ast }g\in G$, $gx\in C$%
.\ Therefore $C$ satisfies the desired conclusions.
\end{proof}

\subsection{Generic homomorphisms between Becker graphs\label%
{Sbs:digraph-homomorphism}}

In this section we use the Becker-embedding game and the orbit continuity
lemma to address Problem \ref{problem} for the class of CLI groups.

\begin{definition}
An equivalence relation $E$ on a Polish space $X$ is CLI-classifiable if it
is Borel reducible to $E_{H}^{Y}$ for some CLI group $H$ and Polish $H$%
-space $Y$.
\end{definition}

We will obtain below an obstruction to CLI-classifiability in terms of the
Becker digraph. This will be based upon the following properties of the
Becker digraph:

\begin{enumerate}
\item the Becker digraph contains only loops in the case of CLI group
actions (Lemma \ref{lemma CLI totally Becker disco}), and

\item a Baire-measurable homomorphism between orbit equivalence relations
induces, after restricting to an invariant dense $G_{\delta }$ set, a
homomorphism at the level of Becker digraphs (Proposition \ref%
{Proposition:digraph-homomorphism}).
\end{enumerate}

\begin{lemma}
\label{lemma CLI totally Becker disco}If $Y$ is a Polish $H$-space and $H$
is a CLI group, then the Becker digraph $\mathcal{B}\left( Y/H\right) $
contains only loops.
\end{lemma}

\begin{proof}
Fix a compatible complete right-invariant metric $d$ on $H$. For a subset $A$
of $H$ we let $\mathrm{diam}\left( A\right) $ be the diameter of $A$ with
respect to $d$. Let $x,y$ be elements of $Y$ with different $H$-orbits. We
show that Player I has a winning strategy in $\mathrm{Emb}(x,y)$. In the $n$%
-th round Player $I$ plays some symmetric open neighborhood $V_{n+1}$ of the
identity of $H$ with $\mathrm{diam}\left( V_{n+1}\right) <2^{-n}$ and an
open neighborhood $U_{n}$ of $x$ such that the sequence $\left( U_{n}\right) 
$ forms a decreasing basis of neighborhoods of $x$. Let $\left( g_{n}\right) 
$ be the sequence of group elements chosen by Player II, and set $%
h_{n}:=g_{n}\cdots g_{0}$. We claim that such a sequence does not satisfy
the winning condition for Player II in the Becker-embedding game. Suppose by
contradiction that this is the case, and hence $\lim_{n}h_{n}y=x$. For every 
$n>m$ we have by right invariance of $d$ that 
\begin{equation*}
d(h_{n},h_{m})=d(g_{n}\cdots g_{m+1},1)\leq d(g_{n},1)+d(g_{n-1},1)+\cdots
+d(g_{m+1},1)<2^{-m}\text{.}
\end{equation*}%
Therefore $h_{n}$ is a $d$-Cauchy sequence with respect to $d$. Since by
assumption $d$ is complete, $h_{n}$ converges to some $h\in H$. From $%
\lim_{n}h_{n}y=x$ and continuity of the action, we deduce that $hy=x$. This
contradicts the assumption that the $H$-orbits of $x$ and $y$ are different.
\end{proof}

Using the orbit continuity\ lemma (Lemma \ref{Lemma:orbit-continuity}) one
can then show that a Baire-measurable homomorphism preserves Becker
embeddability on a comeager set. This is the content of the following
proposition.

\begin{proposition}
\label{Proposition:digraph-homomorphism}Suppose that $G,H$ are Polish
groups, $X$ is a Polish $G$-space, and $Y$ is a Polish $H$-space. Let $%
f:X\rightarrow Y$ be a Baire-measurable $\left( E_{G}^{X},E_{H}^{Y}\right) $%
-homomorphism. Then there exists a $G$-invariant dense $G_{\delta }$ subset $%
X_{0}$ of $X$ such that the function $\left[ f\right] :X_{0}/G\rightarrow
Y/H $, $\left[ x\right] \mapsto \left[ f(x)\right] $ is a digraph
homomorphism from the Becker digraph $\mathcal{B}\left( X_{0}/G\right) $ to
the Becker digraph $\mathcal{B}\left( Y/G\right) $.
\end{proposition}

\begin{proof}
 Let $C$ be a dense $G_{\delta }$ subsets of $X$ obtained from $f$ as in Lemma \ref%
{Lemma:orbit-continuity}. Set $X_{0}:=\left\{ x\in X:\forall ^{\ast }g\in
G,gx\in C\right\} $, which is a $G$-invariant dense $G_{\delta }$ set by 
\cite[Proposition 3.2.5 and Theorem 3.2.7]{gao_invariant_2009}. We claim
that $\left[ f\right] :X_{0}/G\rightarrow Y/H$, $\left[ x\right] \mapsto %
\left[ f(x)\right] $ is a digraph homomorphism from the Becker digraph $%
\mathcal{B}\left( X_{0}/G\right) $ to the Becker digraph $\mathcal{B}\left(
Y/G\right) $.

Fix $x_{0},y_{0}\in X_{0}$ such that $x_{0}\preccurlyeq _{\mathcal{B}}y_{0}$%
. We want to prove that $f(x_{0})\preccurlyeq _{\mathcal{B}}f(y_{0})$.
Observe that $\forall ^{\ast }g\in G$, $gx_{0}\in C\cap X_{0}$. Therefore
after replacing $x_{0}$ with $gx_{0}$ for a suitable $g\in G$ we can assume
that $x_{0}\in C\cap X_{0}$. Let us consider the Becker-embedding game $%
\mathrm{Emb}(f(x_{0}),f(y_{0}))$. At the same time we consider the
Becker-embedding game $\mathrm{Emb}(x_{0},y_{0})$ and use the fact that
Player II has a winning strategy for such a game.

In the first turn of $\mathrm{Emb}(f(x_{0}),f(y_{0}))$, Player I plays an
open neighborhood $\widehat{U}_{1}$ of $f(x_{0})$ and an open neighborhood $%
\widehat{V}_{1}$ of the identity of $H$. Consider an open neighborhood $%
U_{1} $ of $x_{0}$ and an open neighborhood $V_{1}$ of the identity of $G$
such that for any $x\in U_{1}\cap C\cap X_{0}$ and a comeager set of $g\in
V_{1}$ one has that $f\left( gx\right) \in \widehat{V}_{1}f(x)$. Consider
now the round of the game $\mathrm{Emb}(x_{0},y_{0})$ where, in the first
turn, Player I plays the neighborhood $U_{1}$ of $x_{0}$ and the
neighborhood $V_{1}$ of the identity of $G$. Since by assumption Player II
has a winning strategy for $\mathrm{Emb}(x_{0},y_{0})$, we can consider an
element $g_{0}$ of $V_{1}$ which is obtained from such a winning strategy.
By Lemma \ref{Lemma:comeager-Becker}, we can also insist that $g_{0}$
belongs to the comeager set of $g\in V_{1}$ such that $gy_{0}\in U_{1}\cap
C\cap X_{0}$ and $f\left( gy_{0}\right) \in \widehat{V}_{1}f(x)$. We can
then let Player II play, in the first turn of the game $\mathrm{Emb}%
(f(x_{0}),f(y_{0}))$, an element $h_{0}$ of $\widehat{V}_{1}$ such that $%
f\left( g_{0}y_{0}\right) =h_{0}f\left( y_{0}\right) $.

At the $n$-th turn of $\mathrm{Emb}(f(x_{0}),f(y_{0}))$, Player I plays an
open neighborhood $\widehat{U}_{n}$ of $f(x_{0})$ and an open neighborhood $%
\widehat{V}_{n}$ of the identity of $H$. Consider now an open neighborhood $%
U_{n}$ of $x_{0}$ and an open neighborhood $V_{n}$ of the identity of $G$
such that for any $x\in U_{n}\cap C\cap X_{0}$ and a comeager set of $g\in
V_{n}$ one has that $f\left( gx\right) \in \widehat{V}_{n}f(x)$. Let Player
I play, in the $n$-turn of $\mathrm{Emb}(x_{0},y_{0})$, the open
neighborhoods $U_{n}$ of $x_{0}$ and $V_{n}$ of the identity of $G$. Let $%
g_{n-1}\in V_{n}$ be obtained from a winning strategy for Player II. By
Lemma \ref{Lemma:comeager-Becker} we can insist that $g_{n-1}$ belongs to
the comeager set of $g\in V_{n}$ such that $gg_{n-2}\cdots g_{1}g_{0}y\in
U_{n}\cap C\cap X_{0}$ and $f\left( gx\right) \in \widehat{V}_{1}f(x)$.
Therefore we can let Player II play, in the $n$-th turn of the game $\mathrm{%
Emb}(f(x_{0}),f(y_{0}))$, an element $h_{n-1}\in \widehat{V}_{n-1}$ such
that $f\left( g_{n-1}\cdots g_{0}y\right) =h_{n-1}f\left( g_{n-1}\cdots
g_{0}y\right) =h_{n-1}\cdots h_{0}y\in \widehat{U}_{n}$. Such a construction
witness that Player II has a winning strategy for the game $\mathrm{Emb}%
(f(x_{0}),f(y_{0}))$.
\end{proof}

From Lemma \ref{lemma CLI totally Becker disco} and Proposition \ref%
{Proposition:digraph-homomorphism} one can immediately deduce the following
criterion to show that the orbit equivalence relation of a Polish group
action is not Borel reducible to the orbit equivalence relation of CLI group
action.

\begin{theorem}
\label{Theorem:criterion}Suppose that $X$ is a Polish $G$-space. If for any $%
G$-invariant dense $G_{\delta }$ subset $C$ of $X$ there exist $x,y\in C$
with different $G$-orbits such that $x\preccurlyeq _{\mathcal{B}}y$, then
for any $G$-invariant dense $G_{\delta }$ subset $C$ of $X$ the relation $%
E_{C}^{X}$ is not CLI-classifiable.
\end{theorem}

\begin{proof}
Suppose that $H$ is a CLI group, and $Y$ is a Polish $H$-space. Suppose that 
$D$ is a $G$-invariant dense $G_{\delta }$ subset of $X$, and $%
f:D\rightarrow Y$ is a Borel $\left( E_{G}^{D},E_{H}^{Y}\right) $%
-homomorphism. Then by Proposition \ref{Proposition:digraph-homomorphism}
there exists a $G$-invariant dense $G_{\delta }$ subset $C$ of $D$ such that 
$\left[ f\right] :C/G\rightarrow Y/H$ is a digraph homomorphisms for the
Becker digraphs $\mathcal{B}\left( C/G\right) $ and $\mathcal{B}\left(
Y/H\right) $. By assumption there exist elements $x,y$ of $C$ with different 
$G$-orbits such that $x\preccurlyeq _{\mathcal{B}}y$. Therefore $%
f(x)\preccurlyeq _{\mathcal{B}}f(y)$. Since $H$ is CLI we have by Lemma \ref%
{lemma CLI totally Becker disco} that $f(x)$ and $f(y)$ belong to the same $%
H $-orbit. Therefore $f$ is not a reduction from $E_{G}^{D}$ to $E_{H}^{Y}$.
\end{proof}

\subsection{Applications}

Suppose that $E$ is an equivalence relation on a Polish space $X$. Recall
that the Friedman--Stanley jump $E^{+}$ of $E$ \cite[Definition 8.3.1]%
{gao_invariant_2009}---see also \cite{friedman_borel_1989}---is the
equivalence relation on the standard Borel space $X^{\mathbb{N}}$ of
sequences of elements of $X$ defined by $\left( x_{n}\right) E^{+}\left(
y_{n}\right) $ if and only if $\left\{ \left[ x_{n}\right] _{E}:n\in \mathbb{%
N}\right\} =\left\{ \left[ y_{n}\right] _{E}:n\in \mathbb{N}\right\} $.

In particular one can start with the relation $=$ of equality on a perfect
Polish space $X$. The corresponding Friedman--Stanley jump is the relation $%
=^{+}$ on $X^{\omega }$ defined by $\left( x_{n}\right) =^{+}\left(
y_{n}\right) $ if and only if the sequences $\left( x_{n}\right) $ and $%
\left( y_{n}\right) $ have the same range. With respect to Borel
reducibility, $=^{+}$ is the most complicated (essentially) $\boldsymbol{\Pi 
}_{3}^{0}$ equivalence relation \cite[Theorem 12.5.5]{gao_invariant_2009};
see also \cite{hjorth_borel_1998}.

Hjorth has proven in \cite[Theorem 5.19]{hjorth_absoluteness_1998} that $%
=^{+}$ is not Borel reducible to the orbit equivalence relation of a
continuous action of an \emph{abelian} Polish group. As remarked in \cite[%
page 663]{hjorth_absoluteness_1998}, Hjorth's proof uses a metamathematical
argument involving forcing and Stern's absoluteness principle . Similar
methods are used in \cite[Theorem 17.1.3]{kanovei_borel_2008} to prove that $%
=^{+}$ is not Borel reducible to the orbit equivalence relation of a Borel
action of a CLI group. This is obtained as a consequence of a general result
concerning pinned equivalence relations; see \cite[Definition 17.1.2]%
{kanovei_borel_2008}. To our knowledge, the argument below provides the
first entirely classical proof of this result.

Let $\sigma :X^{\mathbb{N}}\rightarrow X^{\mathbb{N}}$ be the \emph{%
unilateral shift }$\left( x_{1},x_{2},\ldots \right) \mapsto \left(
x_{2},x_{3},\ldots \right) $. We consider the restriction of $=^{+}$ to the
dense $G_{\delta }$ subset $Y$ of $X^{\mathbb{N}}$ that consists of
injective sequences. Observe that this is the orbit equivalence relation of
the canonical action of $S_{\infty }$ on $X^{\mathbb{N}}$ obtained by
permuting the indices.

\begin{theorem}
\label{Theorem:countable-set}Let $Z\subset Y$ be a nonempty $S_{\infty }$%
-invariant $G_{\delta }$ set such that $\sigma \left[ Z\right] =Z$. The
restriction of $=^{+}$ to any $S_{\infty }$-invariant dense $G_{\delta }$
subset of $Z$ is not Borel reducible to a Borel action of a CLI group on a
standard Borel space.
\end{theorem}

\begin{proof}
Let $E$ be the restriction of $=^{+}$ to $Z$. As observed before, $E$ is the
orbit equivalence relation of the canonical action $S_{\infty
}\curvearrowright Z\subset Y\subset X^{\mathbb{N}}$ given by permuting the
coordinates. We apply Proposition \ref{Theorem:criterion}. Let $C$ be an $%
S_{\infty }$-invariant dense $G_{\delta }$ subset of $Z$. We need to prove
that there exist $x,y\in C$ with different orbits such that $x\preccurlyeq _{%
\mathcal{B}}y$. For $x=\left( x_{n}\right) \in Y$ we let $\mathrm{Ran}(x)$
be the set $\left\{ x_{n}:n\in \mathbb{N}\right\} $. It is not difficult to
see that, for $x,y\in Y$, $x\preccurlyeq _{\mathcal{B}}y$ if and
only if $\mathrm{Ran}(x)\subset \mathrm{Ran}(y)$. Observe that $\sigma
:Z\rightarrow Z$ is continuous, open, and surjective. Therefore, since $C$
is a dense $G_{\delta }$ subset of $Z$, we have that there exists a comeager
subset $C_{0}$ of $C$ such that, for every $x$ $\in C_{0}$, $\sigma
^{-1}\left( x\right) \cap C$ is a comeager subset of $\sigma ^{-1}\left(
x\right) $; see \cite[Theorem A.1]{melleray_generic_2013}. Pick now $x\in
C_{0}$ and $y\in \sigma ^{-1}\left( x\right) \cap C$. It is clear that $%
x\preccurlyeq _{\mathcal{B}}y$ and $x,y$ lie in different $S_{\infty }$%
-orbits. This concludes the proof.
\end{proof}

We now apply Theorem \ref{Theorem:countable-set} to obtain information about
the orbit equivalence relation of some canonical actions of the unitary
group $\mathcal{U}\left( \mathcal{H}\right) $.\ Let $\mathcal{H}$ be the
separable infinite-dimensional Hilbert space, and let $\mathcal{U}\left( 
\mathcal{H}\right) $ be the group of unitary operators on $\mathcal{H}$.
This is a Polish group when endowed with the weak operator topology; see 
\cite[Proposition I.3.2.9]{blackadar_operator_2006}. The group $\mathcal{U}%
\left( \mathcal{H}\right) $ admits an action by conjugation on itself and on
the space $\mathcal{B}\left( \mathcal{H}\right) _{sa}$ of selfadjoint
operators.

\begin{theorem}
The following relations are not Borel reducible to a Borel action of a CLI
group on a standard Borel space:

\begin{enumerate}
\item unitary equivalence of unitary operators;

\item unitary equivalence of selfadjoint operators.
\end{enumerate}
\end{theorem}

\begin{proof}
As in Theorem \ref{Theorem:countable-set} we consider the equivalence
relation $=^{+}$ on the set $X^{\mathbb{N}}$ of sequences of elements of a
perfect Polish space $X$. Fix an orthonormal basis $\left( e_{n}\right) $ of 
$\mathcal{H}$. Let $X\ $be the circle group $\mathbb{T}$, and $Y\subset 
\mathbb{T}^{\mathbb{N}}$ be the set of injective sequences. The map $%
f:Y\rightarrow \mathcal{U}(\mathcal{H})$ which sends an element $(\lambda
_{n})\in Y$ to the unitary operator 
\begin{equation*}
(e_{n})\mapsto (\lambda _{n}e_{n})
\end{equation*}%
is a Borel reduction from $=^{+}|_{Y}$ to unitary equivalence of unitary
operators. The proof of selfadjoint operators is the same, where one
replaces $\mathbb{T}$ with $\left[ 0,1\right] $.
\end{proof}

\section{A game-theoretic approach to turbulence\label{Section:turbulence}}

\subsection{Hjorth's turbulence theory\label{Subsection:theory}}

Suppose that $\mathcal{L}=\left( R_{i}\right) _{i\in I}$ is a countable
first order relational language, where $R_{i}$ is a relation symbol with
arity $n_{i}$. We denote as above by $\mathrm{Mod}\left( \mathcal{L}\right) $
the Polish $S_{\infty }$-space of $\mathcal{L}$-structures with support $%
\mathbb{N}$. Recall that a Polish group $G$ is called \emph{non-Archimedean }%
if it admits a neighborhood basis of the identity of open subgroups or,
equivalently, it is isomorphic to a closed subgroup of $S_{\infty }$; see 
\cite[Theorem 1.5.1]{becker_descriptive_1996}. A relation $E$ is \emph{%
classifiable by countable structures }if it is Borel reducible to the
isomorphism relation in $\mathrm{Mod}\left( \mathcal{L}\right) $ for some
countable first order relational language $\mathcal{L}$. This is equivalent
to the assertion that $E$ is Borel reducible to the orbit equivalence
relation of a Borel action of a non-Archimedean Polish group $G$ on a
standard Borel space by \cite[Theorem 5.1.11]{becker_descriptive_1996} and 
\cite[Theorem 3.5.2, Theorem 11.3.8]{gao_invariant_2009}.

\emph{Turbulence} is a dynamical condition on a Polish $G$-space $X$ which
is an obstruction of classifiability of $E_{G}^{X}$ by countable structures.
We now recall here the fundamental notions of the theory of turbulence,
developed by Hjorth in \cite{hjorth_classification_2000}. Suppose that $X$
is a Polish $G$-space, $x\in X$, $U$ is a neighborhood of $x$, and $V$ is a
neighborhood of the identity in $G$. The local orbit $\mathcal{O}(x,U,V)$ is
the smallest subset of $U$ with the property that $x\in \mathcal{O}(x,U,V)$,
and if $g\in V$, $x\in \mathcal{O}(x,U,V)$, and $gx\in U$, then $gx\in 
\mathcal{O}(x,U,V)$. A point $x\in X$ is called \emph{turbulent} if it has
dense orbit and, for any neighborhood $U$ of $x$ and neighborhood $V$ of the
identity in $G$, the closure of $\mathcal{O}(x,U,V)$ is a neighborhood of $x$%
. A Polish $G$-space $X$ is \emph{preturbulent} if every point $x\in X$ is
turbulent, and \emph{turbulent} if every point $x\in X$ is turbulent and has
meager orbit.

An equivalence relation $E$ on a Polish space $X$ is generically $S_{\infty
} $-ergodic if, for any Polish $S_{\infty }$-space $Y$ and Baire-measurable $%
\left( E,E_{S_{\infty }}^{Y}\right) $-homomorphism, there exists a comeager
subset of $X$ that is mapped by $f$ to a single $S_{\infty }$-orbit. By \cite%
[Theorem 3.5.2, Theorem 11.3.8]{gao_invariant_2009}, this is equivalent to
the assertion that, for any non-Archimedean Polish group $H$, Polish $H$%
-space $Y$, and Baire measurable $\left( E,E_{H}^{Y}\right) $-homomorphisms,
there exists a comeager subset of $X$ that is mapped by $f$ to a single $H$%
-orbit. The following is the main result in Hjorth's turbulence theory,
providing a dichotomy for preturbulent Polish $G$-spaces.

\begin{theorem}[Hjorth]
\label{Theorem:turbulence}Suppose that $X$ is a preturbulent Polish $G$%
-space. Then the associated orbit equivalence relation $E_{G}^{X}$ is
generically $S_{\infty }$-ergodic. In particular, either $X$ has a dense $%
G_{\delta }$ orbit, or the restriction of $E_{G}^{X}$ to any comeager subset
of $X$ is not classifiable by countable structure.
\end{theorem}

In this section, for each Polish $G$-space $X$, we define a graph structure $%
\mathcal{H}(X/G)$ with domain the quotient $X/G=\{[x]\colon x\in X\}$ of $X$
via the action of $G$. We call this the Hjorth graph associated with the $G$%
-space $X$. An (induced) subgraph of $\mathcal{H}(X/G)$ is of the form $%
\mathcal{H}(C/G)$, where $C$ is an invariant subset of $X$. We view Hjorth's
turbulence theorem as a corollary of the following facts:

\begin{enumerate}
\item $\mathcal{H}(X/G)$ contains only loops if $G$ is non-Archimedean;

\item $\mathcal{H}(X/G)$ is a clique if the action of $G$ on $X$ is
preturbulent;

\item given a Polish $G$-space $X$ and a Polish $H$-space $Y$, a Baire
measurable $(E_{G}^{X},E_{H}^{Y})$-homomorphism $f$ induces, after
restricting to an invariant dense $G_{\delta }$ set, a graph homomorphism
between the corresponding Hjorth graphs.
\end{enumerate}

\subsection{The Hjorth-isomorphism game}

We start by defining a game associated with points of a given Polish $G$%
-space, which captures isomorphism in the case of Polish $S_{\infty }$%
-spaces.

\begin{definition}
\label{Definition:Hjorth-game}Suppose that $X$ is a Polish $G$-space, and $%
x,y\in X$. We consider the \emph{Hjorth-isomorphism game} $\mathrm{Iso}(x,y)$
played between two players as follows. Set $x_{0}:=x$, $y_{0}:=y$, $%
U_{0}^{y}:=X$, and $V_{0}^{y}=G$.

\begin{enumerate}
\item In the first turn, Player I plays an open neighborhood $U_{0}^{x}$ of $%
x_{0}$ and an open neighborhood $V_{0}^{x}$ of the identity in $G$. Player
II replies with an element $g_{0}^{y}$ in $G$.

\item In the second turn, Player I then plays an open neighborhood $%
U_{1}^{y} $ of $y_{1}:=g_{0}^{y}y_{0}$ and an open neighborhood $V_{1}^{y}$
of the identity of $G$, and Player II replies with an element $g_{0}^{x}$ in 
$G$.

\item[(2n+1)] At the $(2n+1)$-st turn, Player I plays an open neighborhood $%
U_{n}^{x}$ of $x_{n}:=g_{n-1}^{x}x_{n-1}$ and an open neighborhood $%
V_{n}^{x} $ of the identity of $G$, and Player II responds with an element $%
g_{n}^{y}$ of $G$.

\item[(2n+2)] At the $(2n+2)$-nd turn, Player I plays an open neighborhood $%
U_{n+1}^{y}$ of $y_{n+1}:=g_{n}^{y}y_{n}$ and an open neighborhood $%
V_{n+1}^{y}$ of the identity of $G$, and Player II responds with an element $%
g_{n}^{x}$ of $G$.
\end{enumerate}

The game proceed in this way, producing sequences $\left( x_{n}\right) $ and 
$\left( y_{n}\right) $ of elements of $X$, sequences $\left(
g_{n}^{x}\right) $ and $\left( g_{n}^{y}\right) $ of elements of $G$,
sequences $\left( U_{n}^{x}\right) $ and $\left( U_{n}^{y}\right) $ of open
subsets of $X$, and sequences $\left( V_{n}^{x}\right) $ and $\left(
V_{n}^{y}\right) $ of open neighborhoods of the identity in $G$. Player II
wins the game if, for every $n\geq 0$,

\begin{itemize}
\item $y_{n+1}\in U_{n}^{x}$ and $x_{n}\in U_{n}^{y}$,

\item $g_{n}^{y}=h_{k}\cdots h_{0}$ for some $k\geq 0$ and $h_{0},\ldots
,h_{k}\in V_{n}^{y}$ such that $h_{i}\cdots h_{0}y_{n}\in U_{n}^{y}$ for $%
i\leq k$,

\item $g_{n}^{x}=h_{k}\cdots h_{0}$ for some $k\geq 0$ and $h_{0},\ldots
,h_{k}\in V_{n}^{x}$ such that $h_{i}\cdots h_{0}x_{n}\in U_{n}^{x}$ for $%
i\leq k$.
\end{itemize}

We write $x\sim _{\mathcal{H}}y$ and we say that $x,y$ are Hjorth-isomorphic
if Player II has a winning strategy for the Hjorth game $\mathcal{H}(x,y)$.
\end{definition}

\begin{remark}
\label{Remark:comeager}If Player II has a winning strategy for the Hjorth
game as described above, then it also has a winning strategy for the same
game with the additional winning conditions that $g_{n}^{x}=h_{k}\cdots
h_{0} $ for some $h_{0},\ldots ,h_{k}$ from a given comeager subset of $%
V_{n}^{x}$ such that $h_{i}\cdots h_{0}x_{n}$ belongs to a given comeager
subset $X_{0}$ of $X$ for $i=0,\ldots ,k$, provided that the set of $h\in G$
such that $hx\in X_{0}$ is comeager. Similarly one can add the winning
conditions that $g_{n}^{y}=h_{k}\cdots h_{0}$ for some $h_{0},\ldots ,h_{k}$
from a given comeager subset of $V_{n}^{y}$ such that $h_{i}\cdots
h_{0}y_{n} $ belongs to a given comeager subset $X_{0}$ of $X$, provided
that the set of $h\in G$ such that $hy\in X_{0}$ is comeager. This can be
proved similarly to Lemma \ref{Lemma:comeager-Becker} using the following
version of the Kuratowski-Ulam theorem: suppose that $X,Y$ are Polish spaces
and $f:X\rightarrow Y$ is a continuous open map. Then a Baire-measurable
subset $A $ of $X$ is comeager if and only if the set $\left\{ y\in Y:A\cap
f^{-1}\left\{ y\right\} \text{ is comeager in }f^{-1}\left\{ y\right\}
\right\} $ is comeager; see \cite[Theorem A.1]{melleray_generic_2013}. One
can then apply this fact to the continuous and open map $G\times
X\rightarrow X$, $\left( g,x\right) \mapsto gx$.
\end{remark}

The relation $\sim _{\mathcal{H}}$ is an equivalence relation on $X$ which
we call \emph{Hjorth isomorphism}. It is clear that Hjorth isomorphism is a
coarsening of the orbit equivalence relation $E_{G}$ on $G$. Furthermore if $%
x\sim _{\mathcal{H}}y$, $x^{\prime }$ belongs to the $G$-orbit of $x$, and $%
y^{\prime }$ belongs to the $G$-orbit of $y$, then $x^{\prime }\sim _{%
\mathcal{H}}y^{\prime }$. Let as before $X/G$ be the space of $G$-orbits of
elements of $X$. The Hjorth-graph $\mathcal{H}(X/G)$ associated with the
Polish $G$-space $X$ is symmetric, reflexive graph on $X/G$ given by
declaring that there exists an edge between the orbit $\left[ x\right] $ of $%
x$ and the orbit $\left[ y\right] $ of $y$ if and only if $x\sim _{\mathcal{H%
}}y$. We call $\mathcal{H}\left( X/G\right) $ the Hjorth graph associated
with the Polish $G$-space $X$. One can similarly define the Hjorth graph $%
\mathcal{H}\left( C/G\right) $ for any invariant subset $C$ of $X$. A
comeager subgraph $\mathcal{G}$ of $\mathcal{H}\left( X/G\right) $ is a
graph of the form $\mathcal{H}\left( C/G\right) $, for some invariant
comeager subset $C$ of $X$.

\subsection{Generic homomorphisms between Hjorth graphs}

We now proceed to the proof of the properties of Hjorth graphs stated at the
end of Subsection \ref{Subsection:theory}. In the following, for a subset $V$
of $G$ and $k\in \mathbb{N}$ let $V^{k}$ be the set of elements of $G$ that
can be written as the product of $k$ elements from $V$.

\begin{lemma}
\label{Lemma:nonarchimedean}Suppose that $H$ is a non-Archimedean Polish
group, and $Y$ is a Polish $H$-space. Then the Hjorth graph $\mathcal{H}%
\left( Y/H\right) $ contains only loops.
\end{lemma}

\begin{proof}
Suppose that $G\ $is a non-Archimedean Polish group. Fix a compatible
complete metric $d$ on $X$, and a compatible complete metric $d_{G}$ on $G$.
We denote by $\mathrm{diam}(A)$ the diameter of a subset $A$ of $X$ with
respect to the metric $d$, and by $\mathrm{cl}(A)$ the closure of $A$.
Suppose that Player II has a winning strategy for the Hjorth-isomorphism
game $\mathrm{Iso}\left( x,y\right) $. We want to show that $x$ and $y$
belong to the same orbit. This can be seen by letting Player I play open
subsets $U_{n}^{x}$ and $U_{n}^{y}$ of $X$ such that $\mathrm{cl}%
(U_{n+1}^{y})\subset U_{n}^{x}$, $\mathrm{cl}(U_{n}^{x})\subset U_{n}^{y}$, $%
\mathrm{diam}(U_{n}^{x})\leq 2^{-n}$, $\mathrm{diam}\left(
U_{n+1}^{y}\right) \leq 2^{-n}$, and open subgroups $V_{n}^{x}$ and $%
V_{n}^{y}$ of $G$ such that 
\begin{eqnarray*}
V_{n}^{x} &\subset &\left\{ g\in G:d_{G}\left( gg_{n-1}^{x}\cdots
g_{0}^{x},g_{n-1}^{x}\cdots g_{0}^{x}\right) <2^{-n}\right\} \\
V_{n}^{y} &\subset &\left\{ g\in G:d_{G}\left( gg_{n-1}^{y}\cdots
g_{0}^{y},g_{n-1}^{y}\cdots g_{0}^{y}\right) <2^{-n}\right\} \text{.}
\end{eqnarray*}%
Let then $\left( x_{n}\right) $ and $\left( y_{n}\right) $ be the sequences
of elements of $X$ and $\left( g_{n}^{x}\right) $ and $\left(
g_{n}^{y}\right) $ be the sequences of elements of $G$ obtained from the
corresponding round of the Hjorth game. Then the assumptions on $U_{n}^{x}$
and $U_{n}^{y}$ guarantee that the sequences $\left( x_{n}\right) $ and $%
\left( y_{n}\right) $ converge to the same point $z$ of $X$. The assumptions
on $V_{n}^{x}$ and $V_{n}^{y}$ guarantee that the sequences $\left(
g_{n}^{x}g_{n-1}^{x}\cdots g_{0}^{x}\right) _{n\in \omega }$ and $\left(
g_{n}^{y}g_{n-1}^{y}\cdots g_{0}^{y}\right) _{n\in \omega }$ converge in $H$
to elements $g_{\infty }^{x}$ and $g_{\infty }^{y}$ such that $g_{\infty
}^{x}x=z$ and $g_{\infty }^{y}y=z$. This shows that $x$ and $y$ belong to
the same orbit.
\end{proof}

\begin{lemma}
\label{Lemma:turbulent}Suppose that $X$ is a preturbulent Polish $G$-space.
Then the Hjorth graph $\mathcal{H}\left(X/G\right) $ is a clique.
\end{lemma}

\begin{proof}
Suppose that $X$ is a preturbulent Polish $G$-space. Fix $x,y\in X$. We want
to prove that Player II has a winning strategy for the Hjorth game $\mathcal{%
H}\left( x,y\right) $. We begin with a preliminary observation. Suppose that 
$z\in X$, $U$ is an open neighborhood of $z$, and $V$ is an open
neighborhood of the identity in $G$. Let $\mathcal{I}(z,U,V)$ be the
interior of the closure of the local orbit $\mathcal{O}(z,U,V)$. Since $z$
is turbulent, $\mathcal{I}(z,U,V)$ contains $z$. It is not difficult to see
that, for any $w\in \mathcal{I}(z,U,V)$, the local orbit $\mathcal{O}(w,%
\mathcal{I}(z,U,V),V)$ is dense in $\mathcal{I}(z,U,V)$. We use this
observation to conclude that Player II has a winning strategy, which we
proceed to define. As in the definition of the Hjorth game, we let $x_{0}=x$%
, $y_{0}=y$, $U_{0}^{y}=X$, and $V_{0}^{y}=G$. At the $\left( 2n+1\right) $%
-st turn Player II plays an element $g_{n}^{y}=h_{k}\cdots h_{0}\in
(V_{n}^{y})^{k}$ for some $k\geq 1$ such that $y_{n+1}=g_{n}^{y}y_{n}\in 
\mathcal{I}(x_{n},U_{n}^{x},V_{n}^{x})$ and $h_{i}\cdots h_{0}y_{n}\in
U_{n}^{y}$ for $i\leq k$, while at the $\left( 2n+2\right) $-nd turn Player
II plays an element $g_{n}^{x}=h_{k}\cdots h_{0}\in \left( V_{n}^{x}\right)
^{k}$ for some $k\geq 1$ such that $x_{n+1}=g_{n}^{x}x_{n}\in \mathcal{I}%
\left( y_{n+1},U_{n+1}^{y},V_{n+1}^{y}\right) $ and $h_{i}\cdots
h_{0}x_{n}\in U_{n}^{x}$ for $i\leq k$. Such a choice is possible at the $1$%
-st turn since $y$ has dense orbit. It is possible at the $\left(
2n+2\right) $-nd turn ($n\geq 0$) since $y_{n+1}\in \mathcal{I}\left(
x_{n},U_{n}^{x},V_{n}^{x}\right) $ and for every $w\in \mathcal{I}\left(
x_{n},U_{n}^{x},V_{n}^{x}\right) $ the local orbit $\mathcal{O}\left( w,%
\mathcal{I}\left( x_{n},U_{n}^{x},V_{n}^{x}\right) ,V_{n}^{x}\right) $ is
dense in $\mathcal{I}\left( x_{n},U_{n}^{x},V_{n}^{x}\right) $. It is
possible at the $\left( 2n+1\right) $-st turn ($n\geq 1$) since $x_{n}\in 
\mathcal{I}\left( y_{n},U_{n}^{y},V_{n}^{y}\right) $ and for any $w\in 
\mathcal{I}\left( y_{n},U_{n}^{y},V_{n}^{y}\right) $ the local orbit $%
\mathcal{O}\left( w,\mathcal{I}\left( y_{n},U_{n}^{y},V_{n}^{y}\right)
,V_{n}^{y}\right) $ is dense in $\mathcal{I}\left(
y_{n},U_{n}^{y},V_{n}^{y}\right) $. This concludes the proof that Player II
has a winning strategy for the Hjorth game $\mathcal{H}\left( x,y\right) $.
\end{proof}

\begin{proposition}
\label{Proposition:graph-homomorphism}Suppose that $G,H$ are Polish groups, $%
X$ is a Polish $G$-space, and $Y$ is a Polish $H$-space. If $f$ is a
Baire-measurable $\left( E_{G}^{X},E_{H}^{Y}\right) $-homomorphism, then
there exists a $G$-invariant dense $G_{\delta }$ subset $X_{0}$ of $X$ such
that the function $X_{0}/G\rightarrow Y/H$, $\left[ x\right] \mapsto \left[
f(x)\right] $ is a homomorphism from the Hjorth graph $\mathcal{H}\left(
X_{0}/G\right) $ to the Hjorth graph $\mathcal{H}\left( Y/H\right) $.
\end{proposition}

\begin{proof}
We proceed as in the proof of Proposition \ref%
{Proposition:digraph-homomorphism}. Let $C$ be dense $G_{\delta }$ subsets
of $X$ obtained from $f$ as in Lemma \ref{Lemma:orbit-continuity}. Set $%
X_{0}:=\left\{ x\in X:\forall ^{\ast }g\in G,gx\in C\right\} $, which is a $%
G $-invariant dense $G_{\delta }$ set by \cite[Proposition 3.2.5 and Theorem
3.2.7]{gao_invariant_2009}. We claim that $X_{0}/G\rightarrow Y/H$, $\left[ x%
\right] \mapsto \left[ f(x)\right] $ is a graph homomorphism from the Hjorth
graph $\mathcal{H}\left( X_{0}/G\right) $ to the Hjorth graph $\mathcal{H}%
\left( Y/H\right) $.

Fix $x_{0},y_{0}\in X_{0}$ such that $x_{0}\sim _{\mathcal{H}}y_{0}$. We
want to prove that $f(x_{0})\sim _{\mathcal{H}}f(y_{0})$. Observe that $%
\forall ^{\ast }g\in G$, $gx_{0}\in C\cap X_{0}$. Therefore after replacing $%
x_{0}$ with $gx_{0}$ for a suitable $g\in G$ we can assume that $x_{0}\in
C\cap X_{0}$. In this case one can define, similarly as in the proof of
Proposition \ref{Proposition:digraph-homomorphism}, a winning strategy for
Player II for $\mathrm{Iso}(f(x_{0}),f(y_{0}))$ from a winning strategy for
Player II for $\mathrm{Iso}(x_{0},y_{0})$ using Remark \ref{Remark:comeager}
and the choice of $C$.
\end{proof}

It is now easy to see that Theorem \ref{Theorem:turbulence} is an immediate
consequence of Lemma \ref{Lemma:nonarchimedean} and Lemma \ref%
{Lemma:turbulent} together with Proposition \ref%
{Proposition:graph-homomorphism}.

\section{Groupoids\label{Section:groupoids}}

\subsection{Polish groupoids\label{Sbs:groupoids}}

The goal of this section is to observe that the proofs above apply equally
well in the setting of Polish groupoids as introduced in \cite%
{ramsay_polish_2000,ramsay_mackey-glimm_1990,lupini_polish_2014}. A \emph{%
groupoid} $G$ is a small category where every morphism (also called arrow)
is invertible. By identifying any object with the corresponding identity
arrow, one can regard the set $G^{0}$ of objects of $G$ as a subset of $G$.
The source and range maps $s,r:G\rightarrow G^{0}$ assign to every arrow in $%
G$ its domain\ (or source) and codomain\ (or range). The set $G^{2}$ of
composable arrows is the set of pairs $(\gamma ,\rho )$ of arrows from $G$
such that $s(\gamma )=r\left( \rho \right) $. Composition of arrows is a
function $G^{2}\rightarrow G$, $(\gamma ,\rho )\rightarrow \gamma \rho $.\
If $A,B\subset G$, then we denote by $AB$ the set $\left\{ \gamma \rho
:(\gamma ,\rho )\in G^{2}\cap \left( A\times B\right) \right\} $. If $x\in
G^{0}$ and $A\subset G$, then we let $Ax:=A\left\{ x\right\} =\left\{ \gamma
\in G:s(\gamma )=x\right\} $ and $xA:=\left\{ x\right\} A=\left\{ \gamma \in
G:r(\gamma )=x\right\} $.

A \emph{Polish groupoid} is a groupoid $G$ endowed with a topology such that

\renewcommand{\labelenumi}{(\arabic{enumi})}

\begin{enumerate}
\item there exists a countable basis $\mathcal{B}$ of Polish open sets,

\item composition and inversion of arrows are continuous and open,

\item the sets $Gx$ and $xG$ are Polish subspaces for every $x\in G^{0}$, and

\item the set of objects $G^{0}$ is a Polish subspace.
\end{enumerate}

A Polish groupoid\emph{\ }is not required to be globally Hausdorff. Many
Polish groupoids arising in the applications, such as the locally compact
groupoids associated with foliations of manifolds, are not Hausdorff; see 
\cite[Chapter 2]{paterson_groupoids_1999}.

Suppose that $H$ is a Polish group. One can associate with any Polish $H$%
-space $X$ a Polish groupoid $H\ltimes X$---the \emph{action\ groupoid}%
---that completely encodes the action. Such a groupoid has the Cartesian
product $H\times X$ as set of arrows (endowed with the product topology),
and $\left\{ \left( 1_{H},x\right) :x\in X\right\} $ as set of objects.
Source and range maps are defined by $s\left( h,x\right) =\left(
1_{H},x\right) $ and $r\left( h,x\right) =\left( 1_{H},hx\right) $.
Composition is given by $\left( h,x\right) \left( h^{\prime },y\right)
=\left( hh^{\prime },y\right) $ whenever $x=h^{\prime }y$. In this way one
can regard continuous actions of Polish groups on Polish spaces as a
particular instance of Polish groupoids. One can also consider continuous
actions of Polish groupoids on Polish spaces, but these can be in turn
regarded as Polish groupoids via a similar construction as the one described
above. The class of Polish groupoids is also closed under taking
restrictions. If $X$ is a $G_{\delta }$ subset of the set of objects of a
Polish groupoid $G$, then the \emph{restriction} $G|_{X}$ is the collection
of arrows of $G$ with source and range in $X$, endowed with the induced
Polish groupoid structure. More information about Polish groupoids can be
found in \cite{lupini_polish_2014}.

Given a Polish groupoid $G$, the orbit equivalence relation $E_{G}$ is the
equivalence relation on $G^{0}$ defined by setting $xE_{G}y$ if and only if $%
x,y$ are source and range of an arrow from $G$. The orbit of an object in $G$
is the $E_{G}$-class of $x$.

\subsection{Turbulence for Polish groupoids\label{Sbs:groupoid-turbulence}}

The notion of (pre)turbulence for Polish groupoid has been considered in 
\cite[Section 4]{hartz_classification_2015}. Suppose that $G$ is a Polish
groupoid, $x$ is an object of $G$, and $U$ is a neighborhood of $x$ in $G$.
The local orbit $\mathcal{O}(x,U)$ is the smallest subset of $U\cap G^{0}$
with the property that $x\in \mathcal{O}(x,U)$, and if $\gamma \in U$ is
such that $s(\gamma )\in \mathcal{O}(x,U)$, then $r(\gamma )\in \mathcal{O}%
(x,U)$. An object $x$ is called turbulent if it has orbit dense in $G^{0}$
and, for any neighborhood $U$ of $x$, the closure of $\mathcal{O}(x,U)$ is a
neighborhood of $x$ in $G^{0}$. A Polish groupoid is preturbulent if every
object is turbulent, and turbulent if every object is turbulent and has
orbit meager in $G^{0}$. It is not difficult to see that these definitions
are consistent with the ones for Polish group actions, when a Polish group
action is identified with its associated action groupoid.

Suppose that $G$ is a Polish groupoid, and $x,y\in G^{0}$ are two objects of 
$G$. The Hjorth-isomorphism game $\mathrm{Iso}(x,y)$ can be defined
similarly as in Definition \ref{Definition:Hjorth-game}. Set $x_{0}:=x$, $%
y_{0}:=y$, $U_{0}^{y}=G$, and $V_{0}^{y}=G$. In this case, in the first turn
Player I plays an open neighborhood $U_{0}^{x}$ of $x_{0}$ in $G$ and Player
II replies with an element $\gamma _{0}^{y}$ of $G$ with $s\left( \gamma
_{0}^{y}\right) =y_{0}$. In the second turn, Player I plays an open
neighborhood $U_{1}^{y}$ of $y_{1}:=r\left( \gamma _{0}^{y}\right) $ in $G$
and an element $\gamma _{0}^{x}$ of $G$ with $s\left( \gamma _{0}^{x}\right)
=x_{0}$. At the $\left( 2n+1\right) $-st turn, Player I plays an open
neighborhood $U_{n}^{x}$ of $x_{n}:=r\left( \gamma _{n-1}^{x}\right) $ in $G$%
, and Player II responds with an element $\gamma _{n}^{y}$ of $G$ with $%
s\left( \gamma _{n}^{y}\right) =y_{n}$. At the $\left( 2n+2\right) $-nd
turn, Player I plays an open neighborhood $U_{n+1}^{y}$ of $y_{n+1}:=r\left(
\gamma _{n}^{y}\right) $ in $G$, and Player II responds with an element $%
\gamma _{n}^{x}$ of $G$.

The game then produces sequences $\left( x_{n}\right) ,\left( y_{n}\right) $
of objects of $G$, sequences $\left( \gamma _{n}^{x}\right) ,\left( \gamma
_{n}^{y}\right) $ of arrows in $G$, and sequences $\left( U_{n}^{x}\right)
,\left( U_{n}^{y}\right) $ of open subsets of $G$. Player II wins the game
if, for every $n\geq 0$,

\begin{itemize}
\item $y_{n+1}\in U_{n}^{x}$ and $x_{n}\in U_{n}^{y}$,

\item $\gamma _{n}^{y}=\rho _{1}^{y}\rho _{2}^{y}\cdots \rho _{k}^{y}$ for
some $k\geq 1$ and $\rho _{i}^{y}\in V_{n}^{x}$ for $i=1,2,\ldots ,k$, and $%
\gamma _{n}^{x}=\rho _{1}^{x}\cdots \rho _{k}^{x}$ for some $k\geq 1$ and $%
\rho _{i}^{x}\in V_{n}^{y}$ for $i=1,2,\ldots ,k$.
\end{itemize}

As in the case of Polish group actions, this defines an equivalence relation 
$\sim _{\mathcal{H}}$ (Hjorth-isomorphism) on the set of objects of $G$, by
letting $x\sim _{\mathcal{H}}y$ whenever Player II has a winning strategy
for the Hjorth-isomorphism game $\mathrm{Iso}(x,y)$. Adding to the winning
conditions in the Hjorth-isomorphism game the requirement that $r(\gamma
_{n}^{x})$ belongs to a given comeager subset $X$ of $G^{0}$ and that $%
\gamma _{n}^{x}$ belongs to a given comeager subset of $Gx_{n}$ yields an
equivalent game, provided that the set of $\gamma \in Gx$ such that $%
r(\gamma )\in X$ is comeager. The same applies to $y$. The
Hjorth-isomorphism relation on $G^{0}$ defines a graph structure $\mathcal{H}%
\left( G\right) $ on the space of $G$-orbits, which we call the Hjorth graph
of $G$. The same proof as Lemma \ref{Lemma:turbulent} shows that if $G$ is a
preturbulent Polish groupoid, then the Hjorth graph $\mathcal{H}\left(
G\right) $ is a clique. The analogue of Lemma \ref{Lemma:orbit-continuity}
for Polish groupoids has been proved in \cite[Lemma 4.5]%
{hartz_classification_2015}. Using this one can then prove the analog of
Proposition \ref{Proposition:graph-homomorphism} and deduce the following
result.

\begin{theorem}
\label{Theorem:groupoid-turbulence}Suppose that $G$ is a preturbulent Polish
groupoid. Then the associated orbit equivalence relation $E_{G}$ is
generically $S_{\infty }$-ergodic.
\end{theorem}

Theorem \ref{Theorem:groupoid-turbulence} recovers \cite[Theorem 4.3]%
{hartz_classification_2015}, and can be seen as the groupoid version of
Theorem \ref{Theorem:turbulence} for Polish groupoids.

Since the operations in the groupoid $G$ are continuous and open, one can
reformulate the Hjorth-isomorphism game $\mathrm{Iso}(x,y)$ as presented
above by letting Player II play open sets rather than groupoid elements. Fix
a countable basis $\mathcal{B}$ of Polish open subsets of $G$. In this
formulation of the game, Player I plays elements $U_{n}^{x},U_{n+1}^{y}$ of $%
\mathcal{B}$ for $n\geq 0$ and player II plays elements $W_{n}^{x},W_{n}^{y} 
$ of $\mathcal{B}$ for $n\geq 0$. The winning conditions are then, setting $%
U_{0}^{y}=G$,

\begin{itemize}
\item $r\left[ W_{n+1}^{y}\right] \subset U_{n}^{x}$ and $r\left[ W_{n}^{x}%
\right] \subset U_{n}^{y}$,

\item $W_{n}^{y}\subset \left( U_{n}^{y}\right) ^{k}$ for some $k\geq 1$ and 
$W_{n}^{x}\subset \left( U_{n}^{y}\right) ^{k}$ for some $k\geq 1$,

\item $y\in s\left[ W_{n}^{y}\cdots W_{0}^{y}\right] $ and $x\in s\left[
W_{n}^{x}\cdots W_{0}^{x}\right] $.
\end{itemize}

Such a version of the Hjorth-isomorphism game fits in the framework of Borel
games as described in \cite[Section 2.A]{kechris_classical_1995}. In fact,
this is an \emph{open }game for Player I and closed for Player II, which
allows one to define an $\omega _{1}$-valued \emph{rank }for strategies of
Player I \cite[Exercise 20.2]{kechris_classical_1995}. Insisting that Player
I only has winning strategies of rank at least $\alpha \in \omega _{1}$ (or
no winning strategy at all) gives a hierarchy of equivalence relations $\sim
_{\alpha }$ indexed by countable ordinals, whose intersection is the Hjorth
isomorphism relation.

\subsection{Becker-embeddings for Polish groupoids}

Similarly as for the Hjorth-isomorphism game, the Becker-embedding game $%
\mathrm{Emb}\left( x,y\right) $ can be defined whenever $x,y$ are objects in
a Polish groupoid $G$. This gives a notion of Becker embedding for objects $%
G $, by letting $x\preccurlyeq _{\mathcal{B}}y$ if and only if Player II has
a winning strategy for $\mathrm{Emb}\left( x,y\right) $. In turn this
induces a digraph structure $\mathcal{B}\left( G\right) $ on the space of $G$%
-orbits.

One can prove the groupoid analog of Proposition \ref%
{Proposition:digraph-homomorphism} in a similar fashion, by replacing Lemma %
\ref{Lemma:orbit-continuity} with \cite[Lemma 4.5]{hartz_classification_2015}%
. One can then deduce the following generalization of Theorem \ref%
{Theorem:criterion} to Polish groupoids.

\begin{theorem}
\label{Theorem:criterion-groupoid}Suppose that $G$ is a Polish groupoid. If
for any invariant dense $G_{\delta }$ subset $C$ of $G^{0}$ there exist $%
x,y\in C$ with different orbits such that $x\preccurlyeq _{\mathcal{B}}y$,
then the orbit equivalence relation $E_{G}$ is not CLI-classifiable.
\end{theorem}

As for the case of the Hjorth-isomorphism game, one can also describe the
Becker-embedding game $\mathrm{Emb}\left( x,y\right) $ for objects $x,y$ in
a Polish groupoid $G$ as an open game for Player I and closed for Player II.
This allows one to define an $\omega _{1}$-valued rank for strategies for
Player I. Again, insisting that Player I only has winning stategies of rank
at least $\alpha \in \omega _{1}$ gives a hierarchy or preorder relations $%
\preccurlyeq _{\alpha }$ indexed by countable ordinals, whose intersection
is the Becker-embeddability preorder.

\providecommand{\MR}[1]{}
\providecommand{\bysame}{\leavevmode\hbox to3em{\hrulefill}\thinspace}
\providecommand{\MR}{\relax\ifhmode\unskip\space\fi MR }
\providecommand{\MRhref}[2]{%
  \href{http://www.ams.org/mathscinet-getitem?mr=#1}{#2}
}
\providecommand{\href}[2]{#2}

\end{document}